\documentclass[12pt, reqno]{amsart}
\usepackage{amsmath, amsthm, amscd, amsfonts, amssymb, graphicx, color}
\usepackage[bookmarksnumbered, colorlinks, plainpages]{hyperref}

\input{mathrsfs.sty}

\newtheorem{theorem}{Theorem}[section]
\newtheorem{lemma}[theorem]{Lemma}

\newtheorem{corollary}[theorem]{Corollary}
\theoremstyle{definition}

\theoremstyle{remark}
\newtheorem{remark}[theorem]{Remark}
\numberwithin{equation}{section}
\begin{document}

\title [  Upper bounds for numerical radius inequalities  ]{Upper bounds for numerical radius inequalities involving off-diagonal operator matrices }

\author[ M. Bakherad and K. Shebrawi  ]{Mojtaba Bakherad$^1$ and Khalid Shebrawi$^2$}

\address{ $^1$Department of Mathematics, Faculty of Mathematics, University of Sistan and Baluchestan, Zahedan, I.R.Iran.}

\email{mojtaba.bakherad@yahoo.com; bakherad@member.ams.org}

\address{ $^2$Department of Mathematics, Al-Balqa' Applied University, Salt, Jordan.}

\email{khalid@bau.edu.jo; shebrawi@gmail.com}


\begin{abstract}
In this paper, we establish some upper bounds for  numerical radius inequalities including of $2\times 2$ operator
matrices and their off-diagonal parts. Among other inequalities, it is shown that if $T=\left[\begin{array}{cc}
 0&X\\
 Y&0
 \end{array}\right]$, then
 \begin{align*}
 \omega^{r}(T)\leq 2^{r-2}\left\|f^{2r}(|X|)+g^{2r}(|Y^*|)\right\|^\frac{1}{2}\left\|f^{2r}(|Y|)+g^{2r}(|X^*|)\right\|^\frac{1}{2}
   \end{align*}
   and
   \begin{align*}
 \omega^{r}(T)\leq 2^{r-2}\left\|f^{2r}(|X|)+f^{2r}(|Y^*|)\right\|^\frac{1}{2}\left\|g^{2r}(|Y|)+g^{2r}(|X^*|)\right\|^\frac{1}{2},
   \end{align*}
where $X, Y$ are bounded linear operators on a Hilbert space ${\mathscr H}$, $r\geq 1$ and $f$, $g$ are nonnegative  continuous  functions on $[0, \infty)$ satisfying the relation $f(t)g(t)=t\,(t\in[0, \infty))$. Moreover, we present some inequalities involving the generalized Euclidean operator radius of operators $T_{1},\cdots,T_{n}$.
\end{abstract} \maketitle
\section{Introduction}
 Let ${\mathbb B}(\mathscr H)$ denote the $C^{*}$-algebra of all bounded linear operators on a Hilbert space ${\mathscr H}$. In the case when ${\rm dim}{\mathscr H}=n$, we identify ${\mathbb B}({\mathscr H})$ with the matrix
algebra $\mathbb{M}_n$ of all $n\times n$ matrices with entries in
the complex field. An operator $A\in{\mathbb B}(\mathscr H)$ is said to be contraction, if $A^*A\leq I$.  The numerical radius of $T\in {\mathbb B}({\mathscr H})$ is defined by
$$\omega(T):=\sup\{\left| \langle Tx, x\rangle\right| : x\in {\mathscr H}, \parallel x \parallel=1\}.$$
It is well known that $\omega(\,\cdot\,)$ defines a norm on ${\mathbb B}({\mathscr H})$, which is equivalent to the usual operator norm. In fact, $\frac{1}{2}\| \,\cdot\, \|\leq \omega(\,\cdot\,) \leq\| \,\cdot\, \|$; see \cite{gof}.
 An important inequality for $\omega(A)$ is the power inequality stating that $\omega(A^n)\leq \omega(A)^n\,\,(n=1,2,\cdots)$. For further information about the properties of numerical radius inequalities we refer the reader to \cite{aA, ando, sheikh} and references therein. Let ${\mathscr H_{1}},{\mathscr H_{2}}$ be Hilbert spaces, and consider the direct sum ${\mathscr H}={\mathscr H_{1}}\oplus{\mathscr H_{2}}$. With respect to this decomposition, every operator $T\in {\mathbb B}({\mathscr H})$ has a $2\times 2$ operator matrix representation $T=[T_{ij}]$ with entries  $T_{ij}\in {\mathbb B}({\mathscr H_{j}}, {\mathscr H_{i}})$, the space of all bounded linear operators from ${\mathscr H_{j}}$ to ${\mathscr H_{i}}\,\,(1\leq i,j \leq 2)$. Operator matrices provide a usual tool for studying Hilbert space operators, which have been extensively studied in the literatures.
Let $A\in {\mathbb B}({\mathscr H_{1}}, {\mathscr H_{1}})$, $B\in {\mathbb B}({\mathscr H_{2}}, {\mathscr H_{1}})$, $C\in {\mathbb B}({\mathscr H_{1}}, {\mathscr H_{2}})$ and $D\in {\mathbb B}({\mathscr H_{2}}, {\mathscr H_{2}})$. The operator $\left[\begin{array}{cc} A&0\\ 0&D \end{array}\right]$ is called the diagonal part of $\left[\begin{array}{cc} A&B\\ C&D \end{array}\right]$ and $\left[\begin{array}{cc}  0&B\\ C&0 \end{array}\right]$ is the off-diagonal part.\\
  The classical Young inequality says that if $p, q>1$ such that $\frac{1}{p}+\frac{1}{q}=1$, then $ab\leq \frac{a^{p}}{p}+\frac{b^{q}}{q}$ for positive real numbers $a, b$. In \cite{FUJ}, the authors showed that a refinement of the scalar Young inequality  as follows $\left(a^{\frac{1}{p}}b^{\frac{1}{q}}\right)^{m}+r_{0}^{m}\left(a^{\frac{m}{2}}-b^{\frac{m}{2}}\right)^{2}\leq\left(\frac{a}{p}+\frac{b}{q}\right)^{m},$
where $r_{0}=\min \{ \frac{1}{p}, \frac{1}{q}\}$ and $m=1, 2,\cdots$. In particular, if $p=q=2$, then
\begin{align}\label{12}
\left(a^{\frac{1}{2}}b^{\frac{1}{2}}\right)^{m}+\left(\frac{1}{2}\right)^{m}\left(a^{\frac{m}{2}}-b^{\frac{m}{2}}\right)^{2}\leq 2^{-m}(a+b)^{m}.
\end{align}
It has been shown in \cite{naj}, that if $T\in {\mathbb B}({\mathscr H})$, then
\begin{align}\label{13}
\omega(T)\leq \frac{1}{2} \| |T| + |T^{*}| \|,
\end{align}
where $|T|=(T^{*}T)^{\frac{1}{2}}$ is the absolute value of $T$. Recently \cite{CAL}, the authors extended this inequality for off-diagonal operator matrices of the form  $T=\left[\begin{array}{cc}
 0&X\\
 Y&0
 \end{array}\right]\in {\mathbb B}({\mathscr H_1\oplus\mathscr H_2})$ as follows
 \begin{align}\label{133}
\omega(T)\leq \frac{1}{2} \left\| |X| + |Y^{*}| \right\|^{\frac{1}{2}}\left\| |X^*| + |Y| \right\|^{\frac{1}{2}}.
\end{align}
Let $T_{1}, T_{2},\cdots,T_{n}\in {\mathbb B}({\mathscr H})$.
The functional $\omega_{p}$ of operators $T_{1},\cdots,T_{n}$ for $p\geq 1$ is defined in \cite{FUJ2} as follows
\begin{align*}
\omega_{p}(T_{1},\cdots,T_{n}):= \sup_{\| x \| = 1} \left(\sum_{i=1}^{n} | \left\langle T_{i}x, x\right\rangle |^{p}\right)^{\frac{1}{p}}.
\end{align*}
If $p=2$, then we have the Euclidean operator radius of $T_{1},\cdots,T_{n}$ which was defined in \cite{pop}.
In \cite{sheikh}, the authors showed that an upper bound for the functional $\omega_{p}$
\begin{align*}
 \omega_{p}^p(T_{1},\cdots,T_{n})\leq {1\over2}\left\|\sum_{i=1}^n\left(f^{2p}(|T_i|)+g^{2p}(
 |T_i^*|)\right)\right\|-\inf_{\|x\|=1} \zeta (x),
 \end{align*}
 where $T_i \in {\mathbb B}({\mathscr H})\,\,(i=1,2,\cdots,n)$, $f$, $g$ are nonnegative  continuous  functions on $[0, \infty)$ such that $f(t)g(t)=t\,(t\in [0, \infty))$, $p\geq 1$   and
 {\footnotesize\begin{align*}
 \zeta (x)=\frac{1}{2}\sum_{i=1}^n\left(\left\langle f^{2p}(|T_i|)x,x\right\rangle^{1\over2}-\left\langle g^{2p}(
 |T_i^*|)x,x\right\rangle^{1\over2}\right)^2.
 \end{align*}}

In this paper,  we show some inequalities involving powers of the numerical radius for off-diagonal parts of $2\times 2$ operator matrices. In particular, we extend inequalities \eqref{13} and \eqref{133} for nonnegative  continuous  functions $f$, $g$  on $[0, \infty)$ such that $f(t)g(t)=t\,(t\in [0, \infty))$.
Moreover, we present some inequalities including the generalized Euclidean operator radius $\omega_{p}$.


\section{main results}
\bigskip To prove our first result, we need the following lemmas.
\begin{lemma}\cite{ROD, yam}\label{1}
Let $X\in {\mathbb B}({\mathscr H})$. Then\newline

$(a)\,\,\omega(X)=\underset{\theta \in
\mathbb{R}
}{\max }\left\Vert \textrm{Re}\left( e^{i\theta }X\right) \right\Vert =\underset{\theta \in
\mathbb{R}
}{\max }\left\Vert \textrm{Im}\left( e^{i\theta }X\right) \right\Vert . $\newline

 $(b)\,\,\omega\left(\left[\begin{array}{cc}
              0&X\\
              X&0
              \end{array}\right]\right)
              = \omega(X).$
\end{lemma}
\bigskip The next lemma follows from the spectral theorem for positive operators and Jensen inequality; see \cite{KIT}.
\begin{lemma}\label{3}
 Let $T\in{\mathbb B}({\mathscr H})$, $ T \geq 0$ and $x\in {\mathscr H}$ such that $\|x\|\leq1$. Then\\
$(a)\,\, \left\langle Tx, x\right\rangle^{r} \leq  \left\langle T^{r}x, x\right\rangle$ for $ r\geq 1.$\\
$(b)\,\,\left\langle T ^{r}x, x\right\rangle  \leq  \left\langle Tx, x\right\rangle^{r}$ for $ 0<r\leq 1$.\\
\end{lemma}
\begin{proof}
Let $ r\geq 1$ and  $x\in {\mathscr H}$ such that $\|x\|\leq1$. Fix $u=\frac{x}{\|x\|}$. Using the McCarty inequality we have
$\left\langle Tu, u\right\rangle^{r} \leq  \left\langle T^{r}u, u\right\rangle$, whence
\begin{align*}
\left\langle Tx, x\right\rangle^{r} &\leq \|x\|^{2r-2} \left\langle T^{r}x, x\right\rangle\\&\leq\left\langle T^{r}x, x\right\rangle\qquad(\textrm {since\,}\|x\|\leq1\,\textrm {and\,}2r-2\geq0).
\end{align*}
Hence, we get the first inequality. The proof of the second inequality is similar.
\end{proof}
\begin{lemma}\cite[Theorem 1]{KIT}\label{5}
Let $T\in{\mathbb B}({\mathscr H})$ and $x, y\in {\mathscr H}$ be any vectors.  If $f$, $g$ are nonnegative  continuous functions on $[0, \infty)$ which are satisfying the relation $f(t)g(t)=t\,(t\in[0, \infty))$, then
\begin{align*}
| \left\langle Tx, y \right\rangle |^2 \leq \left\langle f^2(|T |)x ,x \right\rangle\, \left\langle g^2(| T^{*}|)y,y\right\rangle.
 \end{align*}
 \end{lemma}
\bigskip Now, we are in position to demonstrate the main results of this section by
using some ideas from \cite{CAL, sheikh}.
\begin{theorem}\label{main1}
Let
$T=\left[\begin{array}{cc}
 0&X\\
 Y&0
 \end{array}\right]\in {\mathbb B}({\mathscr H_1\oplus\mathscr H_2})$,  $r\geq 1$ and $f$, $g$ be nonnegative  continuous  functions on $[0, \infty)$ satisfying the relation $f(t)g(t)=t\,(t\in[0, \infty))$. Then
 \begin{align*}
 \omega^{r}(T)\leq 2^{r-2}\left\|f^{2r}(|X|)+g^{2r}(|Y^*|)\right\|^\frac{1}{2}\left\|f^{2r}(|Y|)+g^{2r}(|X^*|)\right\|^\frac{1}{2}
   \end{align*}
   and \begin{align*}
 \omega^{r}(T)\leq 2^{r-2}\left\|f^{2r}(|X|)+f^{2r}(|Y^*|)\right\|^\frac{1}{2}\left\|g^{2r}(|Y|)+g^{2r}(|X^*|)\right\|^\frac{1}{2}.
   \end{align*}
 \end{theorem}
\begin{proof}
Let $\mathbf{x}=\left[\begin{array}{cc}
 x_1\\
 x_2
 \end{array}\right] \in {\mathscr H_1\oplus\mathscr H_2}$ be a unit vector (i.e., $\|x_1\|^2+\|x_2\|^2=1$). Then
 \begin{align*}
 &|\left\langle T\mathbf{x}, \mathbf{x} \right\rangle |^{r}\\&
 =|\left\langle Xx_2, x_1 \right\rangle+\left\langle Yx_1, x_2 \right\rangle |^{r}\\&
 \leq\left(|\left\langle Xx_2, x_1 \right\rangle|+|\left\langle Yx_1, x_2 \right\rangle |\right)^{r} \qquad (\textrm {by the triangular inequality})\\&
 \leq\frac{2^r}{2}\left(|\left\langle Xx_2, x_1 \right\rangle|^r+|\left\langle Yx_1, x_2 \right\rangle |^{r}\right)
 \qquad (\textrm {by the convexity\,} f(t)=t^r)\\&
 \leq\frac{2^r}{2}\Big(\left(\left\langle f^2(|X|)x_2, x_2 \right\rangle^\frac{1}{2}\left\langle g^2(|X^*|)x_1, x_1 \right\rangle^\frac{1}{2}\right)^r
 \\&\qquad+\left(\left\langle f^2(|Y|)x_1, x_1 \right\rangle^\frac{1}{2}\left\langle g^2(|Y^*|)x_2, x_2 \right\rangle^\frac{1}{2} \right)^{r}\Big)
\qquad(\textrm {by Lemma\,\,}\ref{5})\\&\leq\frac{2^r}{2}\left(\left\langle f^{2r}(|X|)x_2, x_2 \right\rangle^\frac{1}{2}\left\langle g^{2r}|X^*|x_1, x_1 \right\rangle^\frac{1}{2}
 +\left\langle f^{2r}(|Y|)x_1, x_1 \right\rangle^\frac{1}{2}\left\langle g^{2r}(|Y^*|)x_2, x_2 \right\rangle^\frac{1}{2}\right)\\&
 \qquad\qquad\qquad\qquad\qquad\qquad\qquad\qquad\qquad (\textrm {by Lemma\,\,\ref{3}(a)})\\&
 \leq\frac{2^r}{2}\left(\left\langle f^{2r}(|X|)x_2, x_2 \right\rangle+\left\langle g^{2r}(|Y^*|)x_2, x_2 \right\rangle\right)^\frac{1}{2}\\&\,\,\,\times
 \left(\left\langle f^{2r}(|Y|)x_1, x_1 \right\rangle+\left\langle g^{2r}(|X^*|)x_1, x_1 \right\rangle\right)^\frac{1}{2}
  \,\,\, (\textrm {by the Cauchy-Schwarz inequality})\\&
  =\frac{2^r}{2}\left\langle (f^{2r}(|X|)+g^{2r}(|Y^*|))x_2, x_2 \right\rangle^\frac{1}{2} \left\langle (f^{2r}(|Y|)+g^{2r}(|X^*|))x_1, x_1 \right\rangle^\frac{1}{2} \\&
 \leq\frac{2^r}{2}\left\|f^{2r}(|X|)+g^{2r}(|Y^*|)\right\|^\frac{1}{2}\left\|f^{2r}(|Y|)+g^{2r}(|X^*|)\right\|^\frac{1}{2}\|x_1\|\|x_2\|\\&
 \leq\frac{2^r}{2}\left\|f^{2r}(|X|)+g^{2r}(|Y^*|)\right\|^\frac{1}{2}\left\|f^{2r}(|Y|)+g^{2r}(|X^*|)\right\|^\frac{1}{2}\left(\frac{\|x_1\|^2+\|x_2\|^2}{2}\right)
 \\&\qquad\qquad\qquad\qquad\qquad\qquad\qquad(\textrm {by the arithmetic-geometric mean inequality})\\&
 =\frac{2^r}{4}\left\|f^{2r}(|X|)+g^{2r}(|Y^*|)\right\|^\frac{1}{2}\left\|f^{2r}(|Y|)+g^{2r}(|X^*|)\right\|^\frac{1}{2}.
 \end{align*}
 Hence, we get the first inequality. Now, applying this fact
 \begin{align}\label{main1eq}
 &|\left\langle T\mathbf{x}, \mathbf{x} \right\rangle |^{r}\nonumber\\&
 =|\left\langle Xx_2, x_1 \right\rangle+\left\langle Yx_1, x_2 \right\rangle |^{r}\nonumber\\&
 \leq\left(|\left\langle Xx_2, x_1 \right\rangle|+|\left\langle Yx_1, x_2 \right\rangle |\right)^{r} \qquad (\textrm {by the triangular inequality})\nonumber\\&
 \leq\frac{2^r}{2}\left(|\left\langle Xx_2, x_1 \right\rangle|^r+|\left\langle Yx_1, x_2 \right\rangle |^{r}\right)
 \qquad (\textrm {by the convexity\,} f(t)=t^r)\nonumber\\&
 \leq\frac{2^r}{2}\left(\left(\left\langle f^2(|X|)x_2, x_2 \right\rangle^\frac{1}{2}\left\langle g^2(|X^*|)x_1, x_1 \right\rangle^\frac{1}{2}\right)^r\right.\nonumber
\\&\qquad +\left.\left(\left\langle g^2(|Y|)x_1, x_1 \right\rangle^\frac{1}{2}\left\langle f^2(|Y^*|)x_2, x_2 \right\rangle^\frac{1}{2} \right)^{r}\right)
\qquad(\textrm {by Lemma\,\,}\ref{5})
 \end{align}
  and  a similar argument to the proof of the first inequality we have the second inequality and this completes the proof of the theorem.
\end{proof}
\bigskip Theorem \ref{main1} includes a special case as follows.
\begin{corollary}\label{corol1}
Let
$T=\left[\begin{array}{cc}
 0&X\\
 Y&0
 \end{array}\right]\in {\mathbb B}({\mathscr H_1\oplus\mathscr H_2})$, $0\leq p\leq1$ and $r\geq1$. Then
 \begin{align*}
 \omega^{r}(T)\leq 2^{r-2}  \left\|  | X |^{2rp} +   | Y^* |^{2r(1-p)} \right\|^{\frac{1}{2}}\left\|| Y |^{2rp} +  | X^* |^{2r(1-p)} \right\|^{\frac{1}{2}}
   \end{align*}
   and
\begin{align*}
 \omega^{r}(T)\leq 2^{r-2}  \left\|  | X |^{2rp} +   | Y^* |^{2rp} \right\|^{\frac{1}{2}}\left\|| Y |^{2r(1-p)} +  | X^* |^{2r(1-p)} \right\|^{\frac{1}{2}}.
   \end{align*}
  \end{corollary}
  \begin{proof}
  The result follows immediately from Theorem \ref{main1} for $f(t)=t^p$ and $g(t)=t^{1-p}\,\,(0\leq p\leq1)$.
  \end{proof}
 \begin{remark} Taking $f(t)=g(t)=t^{\frac{1}{2}}\,(t\in[0,\infty))$ and $r=1$ in Theorem \ref{main1},  we get (see \cite[Theorem 4]{CAL}\label{100})
 \begin{align*}
 \omega(T)\leq \frac{1}{2} \left\|  | X | +   | Y^* | \right\|^{\frac{1}{2}}\left\|| Y | +  | X^* | \right\|^{\frac{1}{2}},
   \end{align*}
   where
$T=\left[\begin{array}{cc}
 0&X\\
 Y&0
 \end{array}\right]\in {\mathbb B}({\mathscr H_1\oplus\mathscr H_2})$.
 \end{remark}

  \bigskip If we put $Y=X$ in Theorem \ref{main1}, then by using Lemma \ref{1}(b) we get an extension of Inequality \eqref{13}.
 \begin{corollary}
Let
$X\in {\mathbb B}({\mathscr H})$,  $r\geq 1$ and $f$, $g$ be nonnegative  continuous  functions on $[0, \infty)$ satisfying the relation $f(t)g(t)=t\,(t\in[0, \infty))$. Then
 \begin{align*}
 \omega^{r}(X)\leq 2^{r-2}\left\|f^{2r}(|X|)+g^{2r}(|X^*|)\right\|
   \end{align*}
   and \begin{align*}
 \omega^{r}(X)\leq 2^{r-2}\left\|f^{2r}(|X|)+f^{2r}(|X^*|)\right\|^\frac{1}{2}\left\|g^{2r}(|X|)+g^{2r}(|X^*|)\right\|^\frac{1}{2}.
 \end{align*}
 \end{corollary}
 \begin{corollary}
Let
$X, Y\in {\mathbb B}({\mathscr H})$ and $0\leq p \leq1$. Then%
\begin{align*}
 \omega^{\frac{r}{2}}\left( XY\right) \leq 2^{r-2}  \left\|  | X |^{2rp} +   | Y^* |^{2r(1-p)} \right\|^{\frac{1}{2}}\left\|| Y |^{2rp} +  | X^* |^{2r(1-p)} \right\|^{\frac{1}{2}}
   \end{align*}
   and
\begin{align*}
\omega^{\frac{r}{2}}\left( XY\right) \leq 2^{r-2}  \left\|  | X |^{2rp} +   | Y^* |^{2rp} \right\|^{\frac{1}{2}}\left\|| Y |^{2r(1-p)} +  | X^* |^{2r(1-p)} \right\|^{\frac{1}{2}}
   \end{align*}
\bigskip for $r\geq 1$.
\end{corollary}
\begin{proof}
It follows from the power inequality $\omega^{\frac{1}{2}}\left( T^{2}\right) \leq \omega\left( T\right)$
that
\begin{equation*}
\omega^{\frac{1}{2}}\left( T^{2}\right) =\omega^{\frac{1}{2}}\left( \left[
\begin{array}{cc}
XY & 0 \\
0 &YX%
\end{array}%
\right] \right) =\max \left\{ \omega^{\frac{1}{2}}\left( XY\right) ,\omega^{\frac{1}{2}%
}\left( YX\right) \right\}.  \label{7}
\end{equation*}%
The required result follows from Corollary \ref{corol1}.
\end{proof}
\begin{corollary}
Let
$X, Y\in {\mathbb B}({\mathscr H})$ and $r\geq 1$%
. Then%
\begin{equation*}
\left\Vert X\pm Y^{\ast }\right\Vert ^{r}\leq 2^{2r-2}\left\Vert \left\vert
X\right\vert ^{r}+\left\vert Y^{\ast }\right\vert ^{r}\right\Vert ^{\frac{1}{2}%
}\left\Vert \left\vert Y\right\vert ^{r}+\left\vert X^{\ast
}\right\vert ^{r}\right\Vert ^{\frac{1}{2}}.
\end{equation*}%
In particular, if $X$\ and $Y$\ are normal operators, then%
\begin{equation}
\left\Vert X\pm Y\right\Vert ^{r}\leq 2^{2r-2}\left\Vert \left\vert
X\right\vert ^{r}+\left\vert Y\right\vert ^{r}\right\Vert .  \label{8}
\end{equation}
\end{corollary}
\begin{proof}
Applying Lemma \ref{1}(a) and  Corollary \ref{corol1} (for $p=\frac{1}{2}$), we
have%
\begin{eqnarray*}
\left\Vert X+Y^{\ast }\right\Vert ^{r} &=&\left\Vert T+T^{\ast }\right\Vert
^{r} \\
&\leq &2^{r}\underset{\theta \in
\mathbb{R}
}{\max }\left\Vert \textrm{Re}\left( e^{i\theta }T\right) \right\Vert ^{r} \\
&=&2^{r}\omega^{r}\left( T\right) \\
&\leq &2^{2r-2}\left\Vert \left\vert X\right\vert ^{r}+\left\vert Y^{\ast
}\right\vert ^{r}\right\Vert ^{\frac{1}{2}}\left\Vert \left\vert
Y\right\vert ^{r}+\left\vert X^{\ast }\right\vert ^{r}\right\Vert ^{\frac{1}{%
2}}
\end{eqnarray*}%
where $T=\left[
\begin{array}{cc}
0 & X \\
Y & 0%
\end{array}%
\right] .$ Similarly,
\begin{eqnarray*}
\left\Vert X-Y^{\ast }\right\Vert ^{r} &=&\left\Vert T-T^{\ast }\right\Vert
^{r} \\
&\leq &2^{r}\underset{\theta \in
\mathbb{R}
}{\max }\left\Vert \textrm{Im}\left( e^{i\theta }T\right) \right\Vert ^{r} \\
&=&2^{r}\omega^{r}\left( T\right) \\
&\leq &2^{2r-2}\left\Vert \left\vert X\right\vert ^{r}+\left\vert Y^{\ast
}\right\vert ^{r}\right\Vert ^{\frac{1}{2}}\left\Vert \left\vert
Y\right\vert ^{r}+\left\vert X^{\ast }\right\vert ^{r}\right\Vert ^{\frac{1}{%
2}}
\end{eqnarray*}%
Hence we get the desired result.
For the particular case, observe that $\left\vert Y^{\ast }\right\vert =$ $%
\left\vert Y\right\vert $ and $|X^{\ast }|=|X|.$
\end{proof}
\begin{remark}
It should be mentioned here that inequality $\left( \text{\ref{8}}\right) $,
which has been given earlier,\ is a generalized form of\ the well-known
inequality (see \cite{Bour}): if $A$\ and $B$\ are normal operators, then%
\begin{equation}\label{normal}
\left\Vert X+Y\right\Vert \leq \left\Vert \left\vert X\right\vert
+\left\vert Y\right\vert \right\Vert .
\end{equation}%
The normality of \ $X$\ and $Y$\ are necessary that means Inequality \eqref{normal} is not true for arbitrary operators $X$\ and $Y$; see \cite{khal}%
\end{remark}
   \bigskip In the next theorem, we show another upper bound for numerical radius involving  off-diagonal operator matrices.
 \begin{theorem}\label{main11}
Let
$T=\left[\begin{array}{cc}
 0&X\\
 Y&0
 \end{array}\right]\in {\mathbb B}({\mathscr H_1\oplus\mathscr H_2})$,  $r\geq 1$ and $f$, $g$ be nonnegative  continuous  functions on $[0, \infty)$ satisfying the relation $f(t)g(t)=t\,(t\in[0, \infty))$. Then
  \begin{align*}
 \omega^{2r}(T)\leq 4^{r-2}\left(\frac{\left\|\left(f^{2r}(|X|)+g^{2r}(|Y^*|)\right)^p\right\|}{p^2}
 +\frac{\left\|\left(f^{2r}(|Y|)+g^{2r}(|X^*|)\right)^q\right\|}{q^2}\right)
   \end{align*}
   and \begin{align*}
 \omega^{2r}(T)\leq 4^{r-2}\left(\frac{\left\|\left(f^{2r}(|X|)+f^{2r}(|Y^*|)\right)^p\right\|}{p^2}
 +\frac{\left(\left\|g^{2r}(|Y|)+g^{2r}(|X^*|)\right)^q\right\|}{q^2}\right),
   \end{align*}
   where $\frac{1}{p}+\frac{1}{q}=1$ and $p\geq1$.
 \end{theorem}
 \begin{proof}
If $\mathbf{x}=\left[\begin{array}{cc}
 x_1\\
 x_2
 \end{array}\right] \in {\mathscr H_1\oplus\mathscr H_2}$ is a unit vector, then by a similar argument to the proof of Theorem
 \ref{main1} we have
 \begin{align}\label{remark1}
 &|\left\langle T\mathbf{x}, \mathbf{x} \right\rangle |^{r}\nonumber\\&
 =|\left\langle Xx_2, x_1 \right\rangle+\left\langle Yx_1, x_2 \right\rangle |^{r}\nonumber\\&
 \leq\left(|\left\langle Xx_2, x_1 \right\rangle|+|\left\langle Yx_1, x_2 \right\rangle |\right)^{r} \qquad (\textrm {by the triangular inequality})\nonumber\\&
 \leq\frac{2^r}{2}\left(|\left\langle Xx_2, x_1 \right\rangle|^r+|\left\langle Yx_1, x_2 \right\rangle |^{r}\right)
 \qquad (\textrm {by the convexity\,} f(t)=t^r)\nonumber\\&
 \leq\frac{2^r}{2}\Big(\left(\left\langle f^2(|X|)x_2, x_2 \right\rangle^\frac{1}{2}\left\langle g^2(|X^*|)x_1, x_1 \right\rangle^\frac{1}{2}\right)^r
 \nonumber\\&\qquad+\left(\left\langle f^2(|Y|)x_1, x_1 \right\rangle^\frac{1}{2}\left\langle g^2(|Y^*|)x_2, x_2 \right\rangle^\frac{1}{2} \right)^{r}\Big)
\qquad(\textrm {by Lemma\,\,}\ref{5})\nonumber\\&\leq\frac{2^r}{2}\left(\left\langle f^{2r}(|X|)x_2, x_2 \right\rangle^\frac{1}{2}\left\langle g^{2r}|X^*|x_1, x_1 \right\rangle^\frac{1}{2}
 +\left\langle f^{2r}(|Y|)x_1, x_1 \right\rangle^\frac{1}{2}\left\langle g^{2r}(|Y^*|)x_2, x_2 \right\rangle^\frac{1}{2}\right)\nonumber\\&
 \qquad\qquad\qquad\qquad\qquad\qquad\qquad\qquad\qquad\qquad (\textrm {by Lemma\,\,\ref{3}(a)})\nonumber\\&
 \leq\frac{2^r}{2}\left(\left\langle f^{2r}(|X|)x_2, x_2 \right\rangle+\left\langle g^{2r}(|Y^*|)x_2, x_2 \right\rangle\right)^\frac{1}{2}\nonumber
 \\&\,\,\,\times\left(\left\langle f^{2r}(|Y|)x_1, x_1 \right\rangle+\left\langle g^{2r}(|X^*|)x_1, x_1 \right\rangle\right)^\frac{1}{2}
  \,\,(\textrm {by the Cauchy-Schwarz inequality})\nonumber\\&
 =\frac{2^r}{2}\left\langle (f^{2r}(|X|)+g^{2r}(|Y^*|))x_2, x_2 \right\rangle^\frac{1}{2} \left\langle (f^{2r}(|Y|)+g^{2r}(|X^*|))x_1, x_1 \right\rangle^\frac{1}{2} \nonumber\\&
  \leq\frac{2^r}{2}\left(\frac{\left\langle (f^{2r}(|X|)+g^{2r}(|Y^*|))x_2, x_2 \right\rangle^\frac{p}{2}}{p}
  +\frac{ \left\langle (f^{2r}(|Y|)+g^{2r}(|X^*|))x_1, x_1 \right\rangle^\frac{q}{2}}{q}\right)\nonumber\\&
  \qquad\qquad\qquad\qquad\qquad\qquad\qquad\qquad\qquad\qquad (\textrm {by the Young inequality}) \nonumber\\&
  \leq\frac{2^r}{2}\left(\frac{\left\langle (f^{2r}(|X|)+g^{2r}(|Y^*|))^px_2, x_2 \right\rangle^\frac{1}{2}}{p}
  +\frac{ \left\langle (f^{2r}(|Y|)+g^{2r}(|X^*|))^qx_1, x_1 \right\rangle^\frac{1}{2}}{q}\right)\nonumber\\&
  \qquad\qquad\qquad\qquad\qquad\qquad\qquad\qquad\qquad\qquad(\textrm {by Lemma\,\,\ref{3}(a)}) \nonumber \\&\leq\frac{2^r}{2}\left(\frac{\left\|\left(f^{2r}(|X|)+g^{2r}(|Y^*|)\right)^p\right\|^\frac{1}{2}}{p}\|x_2\|
  +\frac{ \left\|\left(f^{2r}(|Y|)+g^{2r}(|X^*|)\right)^q\right\|^\frac{1}{2}}{q}\|x_1\|\right).
  \end{align}
  Let $\alpha=\frac{\left\|\left(f^{2r}(|X|)+g^{2r}(|Y^*|)\right)^p\right\|^\frac{1}{2}}{p}$ and $\beta=\frac{\left\|\left(f^{2r}(|Y|)+g^{2r}(|X^*|)\right)^q\right\|^\frac{1}{2}}{q}$. It follows from $$\underset{
  \|x_1\|^2+\|x_2\|^2=1
}{\max}(\alpha\|x_1\|+\beta\|x_2\|)=\underset{\theta \in
  [0,2\pi]
}{\max}(\alpha\sin\theta+\beta\cos\theta)=\sqrt{\alpha^2+\beta^2}$$ and Inequality \eqref{remark1} that we deduce
  \begin{align*}
 &|\left\langle T\mathbf{x}, \mathbf{x} \right\rangle |^{r}\nonumber\\&
  \leq\frac{2^r}{2}\left(\frac{\left\|\left(f^{2r}(|X|)+g^{2r}(|Y^*|)\right)^p\right\|}{p^2}
  +\frac{ \left\|\left(f^{2r}(|Y|)+g^{2r}(|X^*|)\right)^q\right\|}{q^2}\right)^\frac{1}{2}.
 \end{align*}
Taking the supremum over all unit vectors $\mathbf{x}\in \mathscr {H}_1\oplus \mathscr {H}_2$ we get the first inequality. Now, according to inequality \eqref{main1eq}
  and the same argument in the proof of the first inequality, we obtain the second inequality.
\end{proof}

  \begin{remark}
 If $T=\left[\begin{array}{cc}
 0&X\\
 Y&0
 \end{array}\right]\in {\mathbb B}({\mathscr H_1\oplus\mathscr H_2})$ and $\frac{1}{p}+\frac{1}{q}=1$, then by using Theorem \ref{main1} and the Young inequality we obtain the  inequalities
 \begin{align*}
 \omega^{r}(T)\leq 2^{r-2}\left(\frac{\left\|f^{2r}(|X|)+g^{2r}(|Y^*|)\right\|^\frac{p}{2}}{p}+\frac{\left\|f^{2r}(|Y|)+g^{2r}(|X^*|)\right\|^\frac{q}{2}}{q}\right)
   \end{align*}
   and \begin{align*}
 \omega^{r}(T)\leq 2^{r-2}\left(\frac{\left\|f^{2r}(|X|)+f^{2r}(|Y^*|)\right\|^\frac{p}{2}}{p}+\frac{\left\|g^{2r}(|Y|)+g^{2r}(|X^*|)\right\|^\frac{q}{2}}{q}\right),
   \end{align*}
   where   $r\geq 1$ and $f$, $g$ are nonnegative  continuous  functions on $[0, \infty)$ satisfying the relation $f(t)g(t)=t\,(t\in[0, \infty))$. Now,  Theorem \ref{main11} shows some other upper bounds for $\omega(T)$.
\end{remark}
\bigskip In the special case of Theorem \ref{main11} for $Y=X$ and $p=q=2$, we have the next result.
\begin{corollary}
Let
$X\in {\mathbb B}({\mathscr H})$,  $r\geq 1$ and $f$, $g$ be nonnegative  continuous  functions on $[0, \infty)$ satisfying the relation $f(t)g(t)=t\,(t\in[0, \infty))$. Then
\begin{align*}
\omega^{2r}(X)\leq 2^{2r-3}\left\|\left(f^{2r}(|X|)+g^{2r}(|X^*|)\right)^2\right\|
\end{align*}
and \begin{align*}
\omega^{2r}(T)\leq 2^{2r-4}\left(\left\|\left(f^{2r}(|X|)+f^{2r}(|X^*|)\right)^2\right\|+\left(\left\|g^{2r}(|X|)+g^{2r}(|X^*|)\right)^2\right\|\right).
\end{align*}
\end{corollary}
 \bigskip Applying Inequality \eqref{12} we obtain the following theorem.
\begin{theorem}\label{main3}
Let
$T=\left[\begin{array}{cc}
 0&X\\
 Y&0
 \end{array}\right]\in {\mathbb B}({\mathscr H_1\oplus\mathscr H_2})$ and  $f$, $g$ be nonnegative  continuous  functions on $[0, \infty)$  satisfying the relation $f(t)g(t)=t$ $(t\in [0, \infty))$. Then for $r\geq 1$
{\footnotesize\begin{align*}
 \omega^{r}(T)\leq 2^{r-2}\left(\left\|f^{2r}(|X|)+g^{2r}(|Y^*|)\right\|+\left\|f^{2r}(|Y|)+g^{2r}(|X^*|)\right\|\right)-2^{r-2}\inf_{\|(x_1,x_2)\|=1} \zeta (x_1,x_2),
 \end{align*}}
 where
 {\footnotesize\begin{align*}
 \zeta (x_1,x_2)=\left(\left\langle \left(f^{2r}(|X|)+g^{2r}(|Y^*|)\right)x_2,x_2\right\rangle^\frac{1}{2}-\left\langle \left(f^{2r}(|Y|)+g^{2r}(|X^*|)\right)x_1,x_1\right\rangle^\frac{1}{2}\right)^2.
 \end{align*}}
 \end{theorem}
 \begin{proof}
 Let $\mathbf{x}=\left[\begin{array}{cc}
 x_1\\
 x_2
 \end{array}\right]
 \in {\mathscr H_1\oplus\mathscr H_2}$ be  a unit vector. Then
   {\begin{align*}
 |\langle T\mathbf{x}&, \mathbf{x}\rangle|^{r}
 \\& = |\langle Xx_2,x_1\rangle+\langle Yx_1, x_2\rangle|^{r}
 \\&\leq\left(|\langle Xx_2,x_1\rangle|+|\langle Yx_1, x_2\rangle|\right)^{r}\qquad(\textrm {by the triangular inequality})
 \\&\leq\frac{2^r}{2}\left(|\langle Xx_2,x_1\rangle|^r+|\langle Yx_1, x_2\rangle|^r\right)\qquad(\textrm {by the convexity\,} f(t)=t^r)
    \\&\leq\frac{2^r}{2}\left(\langle f^2(|X|)x_2,x_2\rangle^\frac{r}{2}\langle g^2(|X^*|)x_1,x_1\rangle^\frac{r}{2}
  +\langle f^2(|Y|)x_1,x_1\rangle^\frac{r}{2}\langle f^2(|Y^*|)x_2,x_2\rangle^\frac{r}{2}\right)
  \\&\qquad\qquad\qquad\qquad\qquad\qquad\qquad\qquad(\textrm {by Lemma\,\,}\ref{5})
  \\&\leq\frac{2^r}{2}\left(\langle f^{2r}(|X|)x_2,x_2\rangle^\frac{1}{2}\langle g^{2r}(|X^*|)x_1,x_1\rangle^\frac{1}{2}
  +\langle f^{2r}(|Y|)x_1,x_1\rangle^\frac{1}{2}\langle g^{2r}(|Y^*|)x_2,x_2\rangle^\frac{1}{2}\right)
  \\&\leq\frac{2^r}{2}\left(\langle f^{2r}(|X|)x_2,x_2\rangle+\langle g^{2r}(|Y^*|)x_2,x_2\rangle\right)^\frac{1}{2}\left(f^{2r}(|Y|)x_1,x_1\rangle+\langle g^{2r}(|X^*|)x_1,x_1\rangle\right)^\frac{1}{2}
  \\&=\frac{2^r}{2}\langle \left(f^{2r}(|X|)+g^{2r}(|Y^*|)\right)x_2,x_2\rangle^\frac{1}{2}\langle \left(f^{2r}(|Y|)+g^{2r}(|X^*|)\right)x_1,x_1\rangle^\frac{1}{2}\\&
  \leq\frac{2^r}{4}\left(\langle \left(f^{2r}(|X|)+g^{2r}(|Y^*|)\right)x_2,x_2\rangle+\langle \left(f^{2r}(|Y|)+g^{2r}(|X^*|)\right)x_1,x_1\rangle\right)\\&
  \,\,\,\,\,-\frac{2^r}{4}\left(\langle \left(f^{2r}(|X|)+g^{2r}(|Y^*|)\right)x_2,x_2\rangle^\frac{1}{2}-\langle \left(f^{2r}(|Y|)+g^{2r}(|X^*|)\right)x_1,x_1\rangle^\frac{1}{2}\right)^2\\&
  \qquad\qquad\qquad\qquad\qquad\qquad\qquad\qquad(\textrm {by Inequality\,\,}\eqref{12})\\&
  \leq\frac{2^r}{4}\left(\left\|f^{2r}(|X|)+g^{2r}(|Y^*|)\right\|+\left\|f^{2r}(|Y|)+g^{2r}(|X^*|)\right\|\right)\\&
  \,\,\,\,\,-\frac{2^r}{4}\left(\langle \left(f^{2r}(|X|)+g^{2r}(|Y^*|)\right)x_2,x_2\rangle^\frac{1}{2}-\langle \left(f^{2r}(|Y|)+g^{2r}(|X^*|)\right)x_1,x_1\rangle^\frac{1}{2}\right)^2.
  \end{align*}}
 Taking the supremum over all unit vectors $\mathbf{x}=\left[\begin{array}{cc}
 x_1\\
 x_2
 \end{array}\right]
 \in {\mathscr H_1\oplus\mathscr H_2}$ we  get the desired inequality.
 \end{proof}
 \bigskip If we put $Y=X$ in Theorem \ref{main3}, then we get next result.
 \begin{corollary}
Let $X\in {\mathbb B}({\mathscr H})$ and  $f$, $g$ be nonnegative  continuous  functions on $[0, \infty)$  satisfying the relation $f(t)g(t)=t$ $(t\in [0, \infty))$. Then for $r\geq 1$
{\footnotesize\begin{align*}
 \omega^{r}(X)\leq 2^{r-1}\|f^{2r}(|X|)+g^{2r}(|X^*|)\|-2^{r-2}\inf_{\|(x_1,x_2)\|=1} \zeta (x_1,x_2),
 \end{align*}}
 where
 {\footnotesize\begin{align*}
 \zeta (x_1,x_2)=\left(\left\langle \left(f^{2r}(|X|)+g^{2r}(|X^*|)\right)x_2,x_2\right\rangle^\frac{1}{2}-\left\langle \left(f^{2r}(|X|)+g^{2r}(|X^*|)\right)x_1,x_1\right\rangle^\frac{1}{2}\right)^2.
 \end{align*}}
 \end{corollary}
 \begin{remark}
 If
 $\mathbf{x}=\left[\begin{array}{cc}
 x_1\\
 x_2
 \end{array}\right]
 \in {\mathscr H_1\oplus\mathscr H_2}$ is a unit vector, then by using the inequality
 \begin{align*}
 |\langle T\mathbf{x}&, \mathbf{x}\rangle|^{r}
 \\& = |\left\langle Xx_2,x_1\right\rangle+\left\langle Yx_1, x_2\right\rangle|^{r}
 \\&\leq\left(|\left\langle Xx_2,x_1\right\rangle|+|\left\langle Yx_1, x_2\right\rangle|\right)^{r}
 \\&\leq\frac{2^r}{2}\left(|\left\langle Xx_2,x_1\right\rangle|^r+|\left\langle Yx_1, x_2\right\rangle|^r\right)
    \\&\leq\frac{2^r}{2}\left(\left\langle f^2(|X|)x_2,x_2\right\rangle^\frac{r}{2}\left\langle g^2(|X^*|)x_1,x_1\right\rangle^\frac{r}{2}\left\langle g^2(|X|)x_2,x_2\right\rangle^\frac{r}{2}\left\langle f^2(|X^*|)x_1,x_1\right\rangle^\frac{r}{2}\right)
    \end{align*}
     and the same argument in the proof if Theorem \ref{main3} we get the following inequality
 {\footnotesize\begin{align*}
 \omega^{r}(T)\leq \frac{2^r}{4}\left(\|f^{2r}(|X|)+f^{2r}(|Y^*|)\|+\|g^{2r}(|Y|)+g^{2r}(|X^*|)\|\right)-\frac{2^r}{4}\inf_{\|(x_1,x_2)\|=1} \zeta (x_1,x_2),
 \end{align*}}
 where
$T=\left[\begin{array}{cc}
 0&X\\
 Y&0
 \end{array}\right]\in {\mathbb B}({\mathscr H_1\oplus\mathscr H_2})$,   $f$, $g$ are nonnegative  continuous  functions on $[0, \infty)$  satisfying the relation $f(t)g(t)=t$ $(t\in [0, \infty))$, $r\geq 1$ and
 {\footnotesize\begin{align*}
 \zeta (x_1,x_2)=\left(\left\langle \left(f^{2r}(|X|)+f^{2r}(|Y^*|)\right)x_2,x_2\right\rangle^\frac{1}{2}-\left\langle \left(g^{2r}(|Y|)+g^{2r}(|X^*|)\right)x_1,x_1\right\rangle^\frac{1}{2}\right)^2.
 \end{align*}}
 \end{remark}
 \section{Some upper bounds  for $\omega_p$}
\bigskip In this section, we  obtain some upper bounds for $\omega_P$. We first show the following theorem.
\begin{theorem}\label{main4}
Let
$  \widetilde{S}_{i}=\left[\begin{array}{cc}
 A_i&0\\
 0&B_i
 \end{array}\right], \widetilde{T}_{i}=\left[\begin{array}{cc}
 0&X_{i}\\
 Y_{i}&0
 \end{array}\right]$ and $ \widetilde{U}_{i}=\left[\begin{array}{cc}
 C_i&0\\
 0&D_i
 \end{array}\right]
 $ be operators matrices in $ {\mathbb B}({\mathscr H_1\oplus\mathscr H_2})$$\,\,(1\leq i\leq n)$ such that $A_i, B_i, C_i$ and $D_i$ are contractions. Then
\begin{align*}
 \omega_{p}^{p}&({\widetilde{S}}^*_1\widetilde{T}_{1}\widetilde{U}_1, \cdots, {\widetilde{S}}^*_n\widetilde{T}_{n}\widetilde{U}_n)\\&\leq 2^{p-2}\sum_{i=1}^{n}\left\|D^*_if^{2p}(|X_{i}|)D_i+ B^*_ig^{2p}(|Y^*_{i}|)B_i\right\|^\frac{1}{2}\left\|C^*_if^{2p}(|Y_{i}|)C_i+A^*_ig^{2p}(|X^*_{i}|)A_i\right\|^\frac{1}{2}
 \end{align*}
 and
 \begin{align*}
 \omega_{p}^{p}&(\widetilde{S}^*_1\widetilde{T}_{1}\widetilde{U}_1, \cdots, \widetilde{S}^*_n\widetilde{T}_{n}\widetilde{U}_n)\\&\leq 2^{p-2}\sum_{i=1}^{n}\left\|D^*_if^{2p}(|X_{i}|)D_i+ B^*_if^{2p}(|Y^*_{i}|)B_i\right\|^\frac{1}{2}\left\|C^*_ig^{2p}(|Y_{i}|)C_i+A^*_ig^{2p}(|X^*_{i}|)A_i\right\|^\frac{1}{2},
 \end{align*}
 where $p\geq1$.
  \end{theorem}
 \begin{proof}
 For any unit vector $\mathbf{x}=\left[\begin{array}{cc}
 x_1\\
 x_2
 \end{array}\right]
 \in {\mathscr H_1\oplus\mathscr H_2}$ we have

 {\footnotesize\begin{align*}
 &\sum_{i=1}^{n}|\langle T_{i}\mathbf{x}, \mathbf{x}\rangle|^{p}
 \\&=\sum_{i=1}^{n}|\langle A^*_iX_{i}D_ix_2, x_1\rangle+\langle B_i^*Y_{i}C_ix_1, x_2\rangle|^{p}
  \\&\leq \sum_{i=1}^{n}\left(|\langle A^*_iX_{i}D_ix_2, x_1\rangle|+|\langle B_i^*Y_{i}C_ix_1, x_2\rangle|\right)^{p}
  \qquad (\textrm {by the triangular inequality})
 \\&\leq\frac{2^p}{2}\sum_{i=1}^{n}|\langle A^*_iX_{i}D_ix_2, x_1\rangle|^{p}+|\langle B_i^*Y_{i}C_ix_1, x_2\rangle|^{p}
 \qquad (\textrm {by the convexity\,} f(t)=t^p)
 \\&=\frac{2^p}{2}\sum_{i=1}^{n}|\langle X_{i}D_ix_2, A_ix_1\rangle|^{p}+|\langle Y_{i}C_ix_1, B_ix_2\rangle|^{p}
 \\&\leq\frac{2^p}{2}\sum_{i=1}^{n}\langle f^2(|X_{i}|)D_ix_2, D_ix_2\rangle^\frac{p}{2}\langle g^2(|X^*_{i}|)A_ix_1, A_ix_1\rangle^\frac{p}{2}
 \\&\,\,\,\,\,+\langle f^2(|Y_{i}|)C_ix_1, C_ix_1\rangle^\frac{p}{2}\langle g^2(|Y^*_{i}|)B_ix_2, B_ix_2\rangle^\frac{p}{2}\qquad\qquad(\textrm {by Lemma\,\,}\ref{5})
\\&\leq\frac{2^p}{2}\sum_{i=1}^{n}\langle f^{2p}(|X_{i}|)D_ix_2, D_ix_2\rangle^\frac{1}{2}\langle g^{2p}(|X^*_{i}|)A_ix_1, A_ix_1\rangle^\frac{1}{2}
\\&\,\,\,\,\,+\langle f^{2p}(|Y_{i}|)C_ix_1, C_ix_1\rangle^\frac{1}{2}\langle g^{2p}(|Y^*_{i}|)B_ix_2, B_ix_2\rangle^\frac{1}{2}\qquad\qquad(\textrm {by Lemma\,\,\ref{3}(a)})
\\&=\frac{2^p}{2}\sum_{i=1}^{n}\langle D^*_if^{2p}(|X_{i}|)D_ix_2, x_2\rangle^\frac{1}{2}\langle A^*_ig^{2p}(|X^*_{i}|)A_ix_1, x_1\rangle^\frac{1}{2}
\\&\,\,\,\,\,+\langle C^*_if^{2p}(|Y_{i}|)C_ix_1, x_1\rangle^\frac{1}{2}\langle B^*_ig^{2p}(|Y^*_{i}|)B_ix_2, x_2\rangle^\frac{1}{2}
\\&\leq\frac{2^p}{2}\sum_{i=1}^{n}\left(\left\langle D^*_if^{2p}(|X_{i}|)D_ix_2, x_2\rangle+\langle B^*_ig^{2p}(|Y^*_{i}|)B_ix_2, x_2\right\rangle\right)^\frac{1}{2}\\&
\,\,\,\,\,\times\left(\left\langle C^*_if^{2p}(|Y_{i}|)C_ix_1, x_1\right\rangle+\left\langle A^*_ig^{2p}(|X^*_{i}|)A_ix_1, x_1\right\rangle\right)^\frac{1}{2}\,\,(\textrm {by the Cauchy-Schwarz inequality})
\end{align*}

\begin{align*}
&=\frac{2^p}{2}\sum_{i=1}^{n}\Big(\left\langle \left(D^*_if^{2p}(|X_{i}|)D_i+ B^*_ig^{2p}(|Y^*_{i}|)B_i\right)x_2, x_2\right\rangle\Big)^\frac{1}{2}\\&\,\,\,\,\,\times\Big(\left\langle \left(C^*_if^{2p}(|Y_{i}|)C_i+A^*_ig^{2p}(|X^*_{i}|)A_i\right)x_1, x_1\right\rangle\Big)^\frac{1}{2}
\\&\leq\frac{2^p}{2}\sum_{i=1}^{n}\left\|D^*_if^{2p}(|X_{i}|)D_i+ B^*_ig^{2p}(|Y^*_{i}|)B_i\right\|^\frac{1}{2}\left\|C^*_if^{2p}(|Y_{i}|)C_i+A^*_ig^{2p}(|X^*_{i}|)A_i\right\|^\frac{1}{2}\|x_1\|\|x_2\|
\\&=\frac{2^p}{2}\sum_{i=1}^{n}\left\|D^*_if^{2p}(|X_{i}|)D_i+ B^*_ig^{2p}(|Y^*_{i}|)B_i\right\|^\frac{1}{2}\\&\,\,\,\,\,\times\left\|C^*_if^{2p}(|Y_{i}|)C_i+A^*_ig^{2p}(|X^*_{i}|)A_i\right\|^\frac{1}{2}
\left(\frac{\|x_1\|^2+\|x_2\|^2}{2}\right)\\&
=\frac{2^p}{4}\sum_{i=1}^{n}\left\|D^*_if^{2p}(|X_{i}|)D_i+ B^*_ig^{2p}(|Y^*_{i}|)B_i\right\|^\frac{1}{2}\left\|C^*_if^{2p}(|Y_{i}|)C_i+A^*_ig^{2p}(|X^*_{i}|)A_i\right\|^\frac{1}{2}.
  \end{align*}}
Taking the supremum over all unit vectors $\mathbf{x}\in {\mathscr H_1\oplus\mathscr H_2}$ we obtain the first inequality. Using the inequality
 \begin{align*}
 &\sum_{i=1}^{n}|\left\langle T_{i}\mathbf{x}, \mathbf{x}\right\rangle|^{p}
 \\&=\sum_{i=1}^{n}|\left\langle A^*_iX_{i}D_ix_2, x_1\right\rangle+\left\langle B_i^*Y_{i}C_ix_1, x_2\right\rangle|^{p}
  \\&\leq \sum_{i=1}^{n}\left(|\left\langle A^*_iX_{i}D_ix_2, x_1\right\rangle|+|\left\langle B_i^*Y_{i}C_ix_1, x_2\right\rangle|\right)^{p}
  \qquad (\textrm {by the triangular inequality})
 \\&\leq\frac{2^p}{2}\sum_{i=1}^{n}|\left\langle A^*_iX_{i}D_ix_2, x_1\right\rangle|^{p}+|\left\langle B_i^*Y_{i}C_ix_1, x_2\right\rangle|^{p}
 \qquad (\textrm {by the convexity\,} f(t)=t^p)
 \\&=\frac{2^p}{2}\sum_{i=1}^{n}|\left\langle X_{i}D_ix_2, A_ix_1\right\rangle|^{p}+|\left\langle Y_{i}C_ix_1, B_ix_2\right\rangle|^{p}
 \\&\leq\frac{2^p}{2}\sum_{i=1}^{n}\left\langle f^2(|X_{i}|)D_ix_2, D_ix_2\right\rangle^\frac{p}{2}\left\langle g^2(|X^*_{i}|)A_ix_1, A_ix_1\right\rangle^\frac{p}{2}
 \\&\,\,\,\,\,+\left\langle g^2(|Y_{i}|)C_ix_1, C_ix_1\right\rangle^\frac{p}{2}\left\langle f^2(|Y^*_{i}|)B_ix_2, B_ix_2\right\rangle^\frac{p}{2}\\&
\qquad\qquad\qquad\qquad\qquad\qquad\qquad\qquad\qquad\qquad(\textrm {by Lemma\,\,}\ref{5})
 \end{align*}
 and a similar fashion in the proof of the first inequality we reach the second inequality.
 \end{proof}
 \bigskip In the special case of Theorem \ref{main4} for $A_i=B_i=C_i=D_i=I\,\,(1\leq i\leq n)$ we have the next result.
 \begin{corollary}
 Let
 $T_{i}=\left[\begin{array}{cc}
 0&X_{i}\\
 Y_{i}&0
 \end{array}\right]
 \in {\mathbb B}({\mathscr H_1\oplus\mathscr H_2})\,\,(1\leq j\leq n)$. Then
\begin{align*}
 \omega_{p}^{p}(T_{1}, T_{2},\cdots,T_{n})\leq 2^{p-2}\sum_{i=1}^{n}\left\|f^{2p}(|X_{i}|)+g^{2p}(|Y^*_{i}|)\right\|^\frac{1}{2}
\left\|f^{2p}(|Y_{i}|)+g^{2p}(|X^*_{i}|)\right\|^\frac{1}{2}
 \end{align*}
 and
 \begin{align*}
 \omega_{p}^{p}(T_{1}, T_{2},\cdots,T_{n})\leq 2^{p-2}\sum_{i=1}^{n}\left\|f^{2p}(|X_{i}|)+f^{2p}(|Y^*_{i}|)\right\|^\frac{1}{2}
\left\|g^{2p}(|Y_{i}|)+g^{2p}(|X^*_{i}|)\right\|^\frac{1}{2}
 \end{align*}
 for  $p\geq 1$.
 \end{corollary}
 \bigskip If we put $f(t)=g(t)=t^{\frac{1}{2}}\,(t\in[0,\infty))$, then we get the next result.
 \begin{corollary}
 Let
 $T_{i}=\left[\begin{array}{cc}
 0&X_{i}\\
 Y_{i}&0
 \end{array}\right]
 \in {\mathbb B}({\mathscr H_1\oplus\mathscr H_2})\,\,(1\leq j\leq n)$. Then
 \begin{align*}
 \omega_{p}^{p}(T_{1}, T_{2},\cdots,T_{n})\leq 2^{p-2}\sum_{i=1}^{n}\left\||X_{i}|^p+|Y^*_{i}|^p\right\|^\frac{1}{2}
\left\||Y_{i}|^p+|X^*_{i}|^p\right\|^\frac{1}{2}
 \end{align*}
 for  $p\geq 1$.
 \end{corollary}

\begin{theorem}
\label{th1}Let
 $T_{i}=\left[
\begin{array}{cc}
A_{i} & B_{i} \\
C_{i} & D_{i}%
\end{array}%
\right]\in {\mathbb B}({\mathscr H_1}\oplus{\mathscr H_2}) \,\,(1\leq i\leq n)$ and $p\geq 1$. Then
\begin{align*}
\omega _{p}^{p}(T_{1},&\ldots ,T_{n})\\&\leq 2^{-p}\sum_{i=1}^{n}\left(
\omega \left( A_{i}\right) +\omega \left( D_{i}\right) +\sqrt{\left( \omega
\left( A_{i}\right) -\omega \left( D_{i}\right) \right) ^{2}+\left(
\left\Vert B_{i}\right\Vert +\left\Vert C_{i}\right\Vert \right) ^{2}}%
\right) ^{p}.  \label{el2}
\end{align*}%
In particular,
\begin{equation*}
\omega \left(\left[
\begin{array}{cc}
A & B \\
C & D%
\end{array}%
\right] \right)\leq \frac{1}{2}\left( \omega \left( A\right) +\omega \left(
D\right) +\sqrt{\left( \omega \left( A\right) -\omega \left( D\right)
\right) ^{2}+\left( \left\Vert B\right\Vert +\left\Vert C\right\Vert \right)
^{2}}\right) .
\end{equation*}
\end{theorem}

\begin{proof}
Let $\mathbf{x}=\left[
\begin{array}{c}
x_1 \\
x_2%
\end{array}%
\right] $ be a unit vector in ${\mathscr H_1\oplus\mathscr H_2}$. Then

\begin{align*}
\left\vert \left\langle T_{i}\mathbf{x},\mathbf{x}\right\rangle \right\vert & =\left\vert
\left\langle \left[
\begin{array}{cc}
A_{i} & B_{i} \\
C_{i} & D_{i}%
\end{array}%
\right] \left[
\begin{array}{c}
x_1 \\
x_2%
\end{array}%
\right] ,\left[
\begin{array}{c}
x_1 \\
x_2%
\end{array}%
\right] \right\rangle \right\vert \\
& =\left\vert \left\langle \left[
\begin{array}{c}
A_{i}x_1+B_{i}x_2 \\
C_{i}x_1+D_{i}x_2%
\end{array}%
\right] ,\left[
\begin{array}{c}
x_1 \\
x_2%
\end{array}%
\right] \right\rangle \right\vert \\
& =\left\vert \left\langle A_{i}x_1,x_1\right\rangle +\left\langle
B_{i}x_2,x_1\right\rangle +\left\langle C_{i}x_1,x_2\right\rangle +\left\langle
D_{i}x_2,x_2\right\rangle \right\vert \\
& \leq\left\vert \left\langle A_{i}x_1,x_1\right\rangle \right\vert +\left\vert
\left\langle B_{i}x_2,x_1\right\rangle \right\vert +\left\vert \left\langle
C_{i}x_1,x_2\right\rangle \right\vert +\left\vert \left\langle
D_{i}x_2,x_2\right\rangle \right\vert
\end{align*}%
Thus,
\begin{align*}
&\omega _{p}^{p}(T_{1},\ldots ,T_{n})\\& =\sup_{\Vert \mathbf{x}\Vert
=1}\sum_{i=1}^{n}\left\vert \left\langle T_{i}\mathbf{x},\mathbf{x}\right\rangle \right\vert
^{p} \\
& \leq \sup_{\Vert x_1\Vert ^{2}+\Vert x_2\Vert ^{2}=1}\sum_{i=1}^{n}\left(
\left\vert \left\langle A_{i}x_1,x_1\right\rangle \right\vert +\left\vert
\left\langle B_{i}x_2,x_1\right\rangle \right\vert +\left\vert \left\langle
C_{i}x_1,x_2\right\rangle \right\vert +\left\vert \left\langle
D_{i}x_2,x_2\right\rangle \right\vert \right) ^{p} \\
& \leq \sum_{i=1}^{n}\left( \sup_{\Vert x_1\Vert ^{2}+\Vert y\Vert
^{2}=1}\left( \left\vert \left\langle A_{i}x_1,x_1\right\rangle \right\vert
+\left\vert \left\langle B_{i}x_2,x_1\right\rangle \right\vert +\left\vert
\left\langle C_{i}x_1,x_2\right\rangle \right\vert +\left\vert \left\langle
D_{i}x_2,x_2\right\rangle \right\vert \right) \right) ^{p} \\
& \leq \sum_{i=1}^{n}\left( \sup_{\Vert x_1\Vert ^{2}+\Vert x_2\Vert
^{2}=1}\left( \omega \left( A_{i}\right) \left\Vert x_1\right\Vert ^{2}+\omega
\left( D_{i}\right) \left\Vert x_2\right\Vert ^{2}+\left( \left\Vert
B_{i}\right\Vert +\left\Vert C_{i}\right\Vert \right) \left\Vert
x_1\right\Vert \left\Vert x_2\right\Vert \right) \right) ^{p} \\
& =\sum_{i=1}^{n}\left( \sup_{\theta \in \left[ 0,2\pi \right] }\left(
\omega \left( A_{i}\right) \cos^2 \theta +\omega \left( D_{i}\right) \sin^2
\theta +\left( \left\Vert B_{i}\right\Vert +\left\Vert C_{i}\right\Vert
\right) \cos \theta \sin \theta \right) \right) ^{p} \\
& =2^{-p}\sum_{i=1}^{n}\left( \omega \left( A_{i}\right) +\omega \left(
D_{i}\right) +\sqrt{\left( \omega \left( A_{i}\right) -\omega \left(
D_{i}\right) \right) ^{2}+\left( \left\Vert B_{i}\right\Vert +\left\Vert
C_{i}\right\Vert \right) ^{2}}\right) ^{p}.
\end{align*}%
This completes the proof.
\end{proof}
\bigskip For $A_{i}=D_{i}$ and $B_{i}=C_{i}\,\,(1\leq i\leq n)$ we get the following result.
\begin{corollary}
Let $T_{i}=\left[
\begin{array}{cc}
\pm A_{i} & \pm B_{i} \\
\pm B_{i} & \pm A_{i}%
\end{array}%
\right] \ $be an operator matrix with $A_{i},B_{i}\in \mathbb{B}(\mathscr{H}%
) $ $\,(1\leq i\leq n)$. Then for all $p\geq 1$,
\begin{equation*}
\omega _{p}^{p}(T_{1},\ldots ,T_{n})\leq \sum_{i=1}^{n}\left( \omega
\left( A_{i}\right) +\left\Vert B_{i}\right\Vert \right) ^{p}.
\end{equation*}%
In particular, if $A,B\in \mathbb{B}(\mathscr{H})$, then%
\begin{equation*}
\omega \left(\left[
\begin{array}{cc}
\pm A & \pm B \\
\pm B & \pm A%
\end{array}%
\right] \right)\leq \omega \left( A\right) +\left\Vert B\right\Vert .
\end{equation*}
\end{corollary}

\bigskip If we take $B_{i}=C_{i}=0\,\,(1\leq i\leq n)$ in Theorem $\ref{th1}$, then we get the
following inequality.
\begin{corollary}
\label{t1}Let  $T_{i}=\left[
\begin{array}{cc}
A_{i} & 0 \\
0 & D_{i}%
\end{array}%
\right]\in{\mathbb B}({\mathscr H_1\oplus\mathscr H_2}) \,\,(1\leq i\leq n)$. Then for
all $p\geq 1$,
\begin{equation*}
\omega _{p}^{p}(T_{1},\ldots ,T_{n})\leq \sum_{i=1}^{n}\max \left(
\omega ^{p}\left( A_{i}\right) ,\omega ^{p}\left( D_{i}\right) \right) .
\end{equation*}
\end{corollary}

\bigskip For $C_{i}=D_{i}=0\,\,(1\leq i\leq n)$ we obtain a result that generalize
and refine the inequality $\omega\left( \left[
\begin{array}{cc}
A & B \\
0 & 0%
\end{array}%
\right] \right) \leq \omega(A)+\frac{\left\Vert B\right\Vert }{2}.$

\begin{corollary}
Let  $%
T_{i}=\left[
\begin{array}{cc}
A_{i} & B_{i} \\
0 & 0%
\end{array}%
\right]\in{\mathbb B}({\mathscr H_1\oplus\mathscr H_2}) \, \,(1\leq i\leq n)$ and $p\geq 1$.
Then
\begin{equation*}
\omega _{p}^{p}(T_{1},\ldots ,T_{n})\leq 2^{-p}\sum_{i=1}^{n}\left(
\omega \left( A_{i}\right) +\sqrt{\omega ^{2}\left( A_{i}\right) +\left\Vert
B_{i}\right\Vert ^{2}}\right) ^{p}.
\end{equation*}%
In particular,
\begin{equation*}
\omega \left( \left[
\begin{array}{cc}
A & B \\
0 & 0%
\end{array}%
\right] \right) \leq \frac{1}{2}\left( \omega \left( A\right) +\sqrt{%
\omega ^{2}\left( A\right) +\left\Vert B\right\Vert ^{2}}\right) .
\end{equation*}%
\bigskip If we put  $A_{i}=D_{i}=0\,\,(1\leq i\leq n)$, then we deduce
\end{corollary}

\begin{corollary}
Let
$T_{i}=\left[
\begin{array}{cc}
0 & B_{i} \\
C_{i} & 0%
\end{array}%
\right]\in{\mathbb B}({\mathscr H}_1\oplus{\mathscr H}_2)  \,(1\leq i\leq n)$ and $p\geq 1$. The
\begin{equation*}
\omega _{p}^{p}(T_{1},\ldots ,T_{n})\leq 2^{-p}\sum_{i=1}^{n}\left(
\left\Vert B_{i}\right\Vert +\left\Vert C_{i}\right\Vert \right) ^{p}.
\end{equation*}%
In particular, if $B\in \mathbb{B}(\mathscr{H}_{2},\mathscr{H}_{1})$ and $%
C\in \mathbb{B}(\mathscr{H}_{1},\mathscr{H}_{2})$, then%
\begin{equation*}
\omega \left( \left[
\begin{array}{cc}
0 & B \\
C & 0%
\end{array}%
\right] \right) \leq \frac{1}{2}\left( \left\Vert B\right\Vert
+\left\Vert C\right\Vert \right).
\end{equation*}
\end{corollary}

\textbf{Acknowledgement.} The first author would like to thank the Tusi Mathematical Research Group (TMRG).
\bigskip
\bibliographystyle{amsplain}

\end{document}